%% file: main.tex
\documentclass[runningheads,envcountsame]{llncs}
\usepackage[T1]{fontenc}
\usepackage{graphicx}
\usepackage[disable]{todonotes}
\usepackage{amsmath,amssymb,mathtools}
\usepackage{xspace}
\usepackage{url}
\usepackage{csquotes}
\usepackage{algorithm}
\usepackage[noend]{algpseudocode}

\usepackage{tikz}
\usepackage{pgfplots}
\pgfplotsset{compat=1.18}
\input{shared-macros}

\usepackage[backend=biber,style=numeric,doi=false,isbn=false,url=false,giveninits=true,maxbibnames=50]{biblatex}
\newcommand\mathematica{\texttt{Mathematica}\xspace}
\begin{document}
\title{Fast Isotopy Computation for T-Curves}
%
%
\author{
  Zoe Geiselmann\inst{1}\orcidID{0009-0009-9244-4752} \and
  Michael Joswig\inst{1}\orcidID{0000-0002-4974-9659} \and
  Lars Kastner\inst{1}\orcidID{0000-0001-9224-7761} \and
  Konrad Mundinger\inst{2}\orcidID{0009-0007-8059-9090} \and
  Sebastian Pokutta\inst{1,2}\orcidID{0000-0001-7365-3000} \and
  Christoph Spiegel\inst{2}\orcidID{0000-0002-6545-202X} \and
  Marcel Wack\inst{1}\orcidID{0009-0005-0360-8569} \and
  Max Zimmer\inst{2}\orcidID{0009-0007-8683-1030}}
\authorrunning{Geiselmann et al.}
%
\institute{Technische Universit\"{a}t Berlin, Chair of Discrete
Mathematics/Geometry\\
\email{\{geiselmann,joswig,kastner,wack\}@math.tu-berlin.de}
\and
Zuse Institut Berlin, AI in Society, Science, and Technology\\
\email{\{mundinger,pokutta,spiegel,zimmer\}@zib.de}}
\maketitle              
\begin{abstract}
  A T-curve of degree $d$ is given by a regular unimodular triangulation of $d \cdot \Delta_2$ together with a sign distribution on its lattice points.
  By Viro's Patchworking Theorem, this determines the ambient isotopy type (a.k.a. real scheme) of a smooth real plane projective algebraic curve of the same degree.
  We present a near-quadratic time algorithm for extracting that isotopy type from the triangulation and the signs.
  Through a GPU-accelerated implementation, this allows one to compute billions of real schemes per second, enabling exhaustive enumeration at scale.
  This algorithm was essential for our recent construction of all 121 real schemes of degree seven by T-curves.

\end{abstract}

\section{Introduction}

Hilbert's 16th problem asks for a topological classification of real plane projective algebraic curves~\cite{1900:Hilbert}.
A common interpretation of \enquote{topological classification} is classification up to ambient isotopy; the resulting classes are called \emph{real schemes}.
That classification is settled through degree seven, with $56$ real schemes for degree six~\cite{1974:Gudkov:TopologyRealProjectiveVarieties} and $121$ for degree seven~\cite{Viro:1984}; degree eight remains open~\cite{2002:Orevkov:FlexibleMcurvesDeg8}.
Determining the real scheme of a curve given by an explicit polynomial is computationally expensive:
general methods like cylindrical algebraic decomposition~\cite{Collins:1975} or specific planar techniques~\cite{Seidel+Wolpert:2005,Kerber-Sagraloff:2012,FruhbisJoswigKastner:2006} rely on polynomial system solving, which is costly.
For instance, in determining real schemes of degree-six curves in \mathematica, Kaihnsa et al.~\cite{2019:KaihnsaEtAl:64curves} found that large coefficients arising from classical constructions made the computation prohibitively slow.

Combinatorial patchworking, introduced by Viro~\cite{Viro:1984,2006:Viro:PatchworkingRealAlgebraicVarieties}, provides a discrete setting in which this computation becomes more tractable.
A regular unimodular triangulation of the dilated simplex $d \cdot \Delta_2$ together with a sign distribution on its lattice points determines a real algebraic curve (a \emph{T-curve}) whose topology depends only on this combinatorial data.
The trade-off is expressiveness: not every real scheme is realizable as a T-curve~\cite{1993:Itenberg:Ragsdale}, although through degree seven all nonempty ones are~\cite{2026:GeiselmannEtAl:121PatchworkedDeg7}.
Viro also describes a procedure for recovering the topology of a patchworked curve~\cite[Algorithms~1.4.A/C/E]{2006:Viro:PatchworkingRealAlgebraicVarieties}: a sequence of geometric surgery operations --- cutting, gluing, and contracting polygons --- that yields a topological model of the curve embedded in $\RP^2$, from which the real scheme can be read off by inspection.
This procedure operates on continuous objects and specifies no discrete data structures, pseudocode, or complexity bounds.

Existing computational tools address aspects of patchworking topology.
The Combinatorial Patchworking Tool~\cite{CombinatorialPatchworkingTool} visualizes patchworked plane curves and counts their connected components.
\texttt{Viro.sage}~\cite{Viro.sage} and \polymake~\cite{2020:JoswigVater:RealTropicalHyperfacesPolymake} compute homological invariants of patchworked hypersurfaces in arbitrary dimension,%
\footnote{\texttt{Viro.sage} computes integral homology via Smith normal form on an explicit simplicial complex; \polymake computes $\mathbb{Z}/2$-Betti numbers via rank computation over $\GF{2}$ on a cellular chain complex derived directly from the combinatorial data.}
which for plane curves reduce to connected component counts and carry no nesting information.
On a related but distinct invariant (the \emph{complex scheme}), De~Loera and Wicklin~\cite{1998:DeLoera-Wicklin} describe a combinatorial algorithm for the dividing type of a T-curve (Type~I vs.\ Type~II, credited to Itenberg--Viro), which represents the closest prior algorithmic work on patchwork-derived invariants.
To the best of our knowledge, no prior published work provides a competitive algorithm that extracts the full real scheme from the combinatorial data.

\medskip\noindent\textbf{Contributions.}
We describe an algorithm that computes the real scheme of a combinatorial patchwork in near-quadratic time in the degree using union-find for connected components, traversal of the region adjacency graph, and rooted-tree canonicalization for the output (Section~\ref{sec:algorithm}).
This enables exhaustive enumeration of patchworks with a given triangulation, a process that was essential for the recent
verification that all 121 real schemes in degree seven are realizable~\cite{2026:GeiselmannEtAl:121PatchworkedDeg7}.
A publicly available C++ implementation is available at \url{github.com/polymake/libisotopy}.
Since the algorithm requires only fixed-size, thread-local state, a GPU-accelerated implementation can classify billions of patchworks per second.
Section~\ref{sec:enumeration} presents results obtained with these implementations for degrees up to eight.

\section{Algorithm}\label{sec:algorithm}

We recall Viro's combinatorial patchwork construction, which defines a
piecewise-linear (PL) curve in the real projective plane $\RP^2$ from purely
combinatorial data, and then present an
algorithm that computes the real scheme of this curve.
The \emph{lattice point set} of degree $d$ is
\begin{equation}\label{eq:lattice}
  A \;=\; d \cdot \Delta_2 \cap \ZZ^2
    \;=\; \bigl\{(i,j) \in \ZZ^2 : i,j \geq 0,\; i+j \leq d\bigr\} \,,
\end{equation}
where $\Delta_2 = \conv\{(0,0),(1,0),(0,1)\}$ is the standard triangle;
$A$ has $\tfrac{1}{2}(d{+}1)(d{+}2)$ elements.
A triangulation $\cT$ of $A$, represented by its edge list, is
\emph{unimodular} if each triangle has Euclidean area~$\tfrac{1}{2}$; equivalently,
every point in~$A$ occurs as a vertex of~$\cT$.
It is \emph{regular} if there exists a lifting function
$\omega \colon A \to \ZZ$ such that $\cT$ is the projection of the lower
convex hull of the lifted points
$\{(i,j,\omega(i,j)) : (i,j) \in A\}$.

Given a triangulation $\cT$ and a \emph{sign distribution}
$\sigma \colon A \to \GF{2}$, the patchworking
construction~\cite[Algorithms~1.4.A/C/E]{2006:Viro:PatchworkingRealAlgebraicVarieties}
proceeds as follows.
Reflecting $\cT$ into four quadrants yields a triangulation $\cT^\diamond$ of
the \emph{diamond}
$A^\diamond = \{(i,j) \in \ZZ^2 : |i|+|j| \leq d\}$, and
$\sigma$ extends from $A$ to $A^\diamond$ by the rule
\begin{align}\label{eq:signs}
  \sigma(i,-j)  &\equiv \sigma(i,j) + j \pmod{2}\,,          \notag \\
  \sigma(-i,j)  &\equiv \sigma(i,j) + i \pmod{2}\,,          \notag \\
  \sigma(-i,-j) &\equiv \sigma(i,j) + i + j \pmod{2}
\end{align}
for $(i,j) \in A$; this rule is forced by the parity structure of monomials
$x^i y^j$ under coordinate reflections.
Identifying antipodal boundary points, $(i,d{-}i) \sim -(i,d{-}i)$ and
$(i,i{-}d) \sim -(i,i{-}d)$ for $0 \leq i \leq d$, yields a cell
decomposition $\cS$ of $\RP^2$.
The \emph{patchworked curve} $\cC(\cT,\sigma)$ is obtained by connecting
the midpoints of edges that separate opposite signs;
this produces a PL~curve in the first barycentric subdivision of~$\cS$.
Combinatorially, $\cC(\cT,\sigma)$ is the subgraph of the dual graph
of~$\cS$ formed by edges whose endpoints have distinct signs.
A real algebraic curve \emph{ambient isotopic} (deformable by a continuous
motion of $\RP^2$) to $\cC(\cT,\sigma)$ is called a \emph{T-curve}.
Viro established that such a curve exists whenever $\cT$ is regular.
We additionally require $\cT$ to be unimodular, which ensures that the
resulting curve is smooth.

\input{tikz/fig-pipeline}

\begin{theorem}[{Viro~\cite{Viro:1984,2006:Viro:PatchworkingRealAlgebraicVarieties}; cf.~\cite[Theorem~1]{2026:GeiselmannEtAl:121PatchworkedDeg7}}]%
  \label{thm:patchworking}
  Let $\cT$ be a regular and unimodular triangulation of~$A$ with lifting
  function $\omega \colon A \to \ZZ$, and let $\sigma \colon A \to \GF{2}$.
  Then there exists $t_0 > 0$ such that for all $t \in (0,t_0]$ the curve
  $V_\RR(f) \subset \RP^2$, where
  $f = \sum_{(i,j) \in A} (-1)^{\sigma(i,j)}\, t^{\omega(i,j)}\,
    x^i y^j z^{d-i-j}$,
  is ambient isotopic to $\cC(\cT,\sigma)$.
\end{theorem}


Our goal is to compute the \emph{real scheme} of $\cC(\cT,\sigma)$ from the
combinatorial data $(\cT,\sigma)$.
The patchworked curve is a disjoint union of simple closed curves in~$\RP^2$.
Each of these is either an \emph{oval} (if it separates $\RP^2$) or a
\emph{pseudoline} (if it does not).
For even~$d$, every curve component is an oval; for odd~$d$, there is exactly
one
pseudoline~\cite[Proposition~3]{2026:GeiselmannEtAl:121PatchworkedDeg7}.
By Harnack's theorem~\cite{1876:Harnack}, the total number of connected components is at
most $M = \tfrac{1}{2}(d{-}1)(d{-}2) + 1$.

We call the connected components of $\RP^2 \setminus \cC(\cT,\sigma)$
\emph{regions}.
Since the ovals are pairwise disjoint and each separates $\RP^2$ into an
interior and an exterior, their nesting defines a tree structure on the
regions.
This is expressed by the \emph{region adjacency graph}, which has one vertex
per region and one edge per connected component of the curve, connecting the
two regions on either side (a self-loop when both sides coincide, as for the
pseudoline).
The interior of each oval is a disk, and the exterior region is a Möbius strip, which
serves as the root of this tree.
For even~$d$, this is the exterior of all ovals; for odd~$d$, it is the
region whose closure contains the pseudoline.
The \emph{real scheme} captures the ambient isotopy class of
$\cC(\cT,\sigma)$ in $\RP^2$ through the isomorphism class of this rooted tree, encoded in \emph{Rohlin--Viro notation}:
$\scheme{0}$ denotes the empty scheme (even~$d$), $\scheme{k}$ denotes $k$
unnested ovals, $\scheme{X \sqcup Y}$ the disjoint union of two subschemes $X$ and $Y$,
$\scheme{k\scheme{X}}$ denotes $k$ ovals all of whose interiors have scheme~$X$, and $J$ marks
the unique pseudoline for odd~$d$.


Recall that $\cC(\cT,\sigma)$ arises as a subgraph in the first barycentric subdivision of~$\cS$.
Its connected components are easily identified.
To derive the real scheme, it remains to determine which of them are separating and to distinguish the interior from the exterior of each oval; one way to accomplish this is by $(\operatorname{mod} 2)$ homology computations.
A na\"{\i}ve implementation has an estimated complexity of $O(d^5)$: there are $O(d^2)$ connected components; each requires one homology computation, e.g., via mod 2 Gauß elimination in $O(d^3)$ time.


Our method for extracting the real scheme from $(\cT,\sigma)$ works on $\cT^\diamond$ directly.
We call an edge of $\cT^\diamond$ \emph{crossed} if its endpoints have
distinct signs, and define a \emph{component} of $(\cT,\sigma)$ as a
connected component of the subgraph induced by the non-crossed edges; each
component is monochromatic by construction.
The crossed edges define a first notion of adjacency on the components.
To pass from $\RR^2$ to $\RP^2$, we identify each of the $4d$ boundary
vertices of $A^\diamond$ with its antipodal partner, merging any two
components that share such a pair; this defines a second, antipodal notion
of adjacency.
The coarsened partition of components coincides with the regions of $\RP^2
\setminus \cC(\cT,\sigma)$ defined above.%
\footnote{Note that antipodal boundary vertices have opposite signs if and only if $d$ is odd.}
The non-orientable root region is detected as follows: for odd~$d$, it is the
unique region containing two components that are crossed-edge adjacent; for
even~$d$, it is the unique region whose components form an odd cycle under
antipodal adjacency.
Algorithm~\ref{alg:real-scheme} summarizes the procedure;
see Figure~\ref{fig:pipeline} for two examples.

\begin{algorithm}[t]
\caption{Real scheme of a combinatorial patchwork.}\label{alg:real-scheme}
\begin{algorithmic}[1]
\Require Edge list $E(\cT)$, sign distribution
  $\sigma \colon A \to \GF{2}$, degree~$d$
\Ensure Canonical real scheme of $\cC(\cT,\sigma)$ in Viro notation
\Statex
\Statex \textbf{Signed reflected triangulation.}
\State Construct $A^\diamond$ and $E(\cT^\diamond)$ by reflecting into four
  quadrants
\State Extend $\sigma$ to $A^\diamond$ via~\eqref{eq:signs}
\Statex
\Statex \textbf{Components and regions.}
\State Compute connected components of same-sign vertices in $\cT^\diamond$
\State Coarsen into regions by identifying antipodal boundary pairs
\Statex
\Statex \textbf{Root and real scheme.}
\State Build region adjacency graph from crossed edges of $\cT^\diamond$
\If{$d$ is even}
  \State $r \gets$ unique region with an odd cycle under antipodal adjacency
\Else
  \State $r \gets$ unique region containing two crossed-edge-adjacent
    components
\EndIf
\State \Return canonical form of the region adjacency graph rooted at~$r$
\end{algorithmic}
\end{algorithm}

\begin{theorem}\label{thm:algorithm}
  Algorithm~\ref{alg:real-scheme} computes the real scheme of
  $\cC(\cT,\sigma)$ in $O(d^2 \cdot \alpha(d^2))$ time, where $\alpha$
  denotes the inverse Ackermann function.
\end{theorem}

\begin{proof}
Removing the crossed edges from $\cT^\diamond$ partitions its vertices into
same-sign connected components; after antipodal boundary identification these
induce the regions of $\RP^2 \setminus
\cC(\cT,\sigma)$~\cite[{\S}2.4]{2026:GeiselmannEtAl:121PatchworkedDeg7}.
Each crossed edge connecting vertices in two distinct regions corresponds to
an oval separating those regions (or the pseudoline for odd~$d$), so the
graph built in Line~5 is the region adjacency graph defined above.
For root detection: when $d$ is odd, the pseudoline is nonseparating, so both
sides belong to the same region, producing the self-loop that identifies the
root; when $d$ is even, every curve component is an oval, and the
non-orientable root region is detected as the unique region whose components
form an odd cycle under antipodal identification: each identification reverses
orientation, so an odd number of reversals makes the region non-orientable.
Since the region adjacency graph is a tree (with a self-loop only at the root
for odd~$d$), traversal from the root and canonicalization recover the real
scheme.
Neither regularity nor unimodularity of $\cT$ is needed for this computation,
only for the algebraic existence guarantee of Theorem~\ref{thm:patchworking}.

For complexity, the reflected triangulation has $O(d^2)$ vertices and edges.
Union-find with path compression and union by rank~\cite{1975:Tarjan:UnionFind}
(Lines~3--4) runs in
$O(d^2 \cdot \alpha(d^2))$ time; Line~4 adds $O(d)$ union operations for
boundary pairs.
Building the region adjacency graph (Line~5) scans $O(d^2)$ edges; root
identification (Lines~6--9) takes $O(d)$ time.
The rooted region adjacency graph has at most $M + 1$ nodes for even~$d$
and at most $M$ for odd~$d$ (where $M$ is given by Harnack's bound), and
is canonicalized in linear time, e.g., via the algorithm in~\cite{1974:AhoHopcroftUllman}.
The total is $O(d^2 \cdot \alpha(d^2))$, dominated by union-find.
\hfill$\square$
\end{proof}

\begin{table}[t]
	\centering
	\caption{Search space parameters by degree.
		$|E(\cT^\diamond)|$ the number of edges in the reflected graph;
		$M$ is Harnack's bound.
		Regular unimodular triangulations up to $\Sym{3}$-symmetry (symmetric representatives,
		i.e., those fixed by $(x,y) \mapsto (y,x)$, in parentheses).
		Sign distributions up to the $(\ZZ/2)^3$-action, giving
		$2^{|A|-3}$ per triangulation.}\label{tab:search-space}
	\smallskip
	\begin{tabular}{r@{\quad}r@{\quad}r@{\quad}r@{\quad}r@{\quad}r}
		\hline
		$d$ & $|A|$ & $|E(\cT^\diamond)|$ & $M$
		& triangulations (symmetric) & sign distributions \\
		\hline
		\\[-1.8ex]
		2 & 6   & 28  & 1  & 2 \,(2)              & 8       \\
		3 & 10  & 60  & 2  & 18 \,(7)             & 128     \\
		4 & 15  & 104 & 4  & 1\,278 \,(74)        & 4\,096  \\
		5 & 21  & 160 & 7  & 561\,885 \,(1\,194)  & 262K    \\
		6 & 28  & 228 & 11 & 1\,198\,202\,590 \,(62\,960)       & 34M  \\
		7 & 36  & 308 & 16 & {--} \,(4\,728\,133)  & 8.6B \\
		8 & 45  & 400 & 22 & {--} \,(1\,199\,795\,773) & 4.4T \\
		\hline
	\end{tabular}
\end{table}

\section{Exhaustive enumeration}\label{sec:enumeration}

Algorithm~\ref{alg:real-scheme} classifies a single pair $(\cT,\sigma)$.
In order to explore the real schemes which are realizable as T-curves for a given
degree~$d$ it is useful to enumerate triangulations and sign
distributions systematically.
Under a natural equivalence relation, the number of sign distributions is reduced by a factor of eight, and triangulation orbits generally contain six elements, three for symmetric triangulations~\cite[Section~3]{2026:GeiselmannEtAl:121PatchworkedDeg7}.%
\footnote{Fixing a triangulation and sweeping over sign distribution representatives yields all types realizable by that orbit of triangulations. The converse --- fixing a sign distribution and sweeping over triangulation representatives --- does not, since the triangulation equivalence also acts on the signs.}
Table~\ref{tab:search-space} summarizes the resulting counts.
This explains our predominant search mode: check all sign distributions for one or a few triangulations.
Here, orchestration overhead is minimal and throughput is near the compute floor; as a side benefit, this yields the distribution of the number of ovals across all sign distribution classes (Figure~\ref{fig:distributions}).
For degrees where checking all triangulations is not feasible, we employed a second mode: check all symmetric triangulations for one or a few sign distributions, at the cost of generating triangulations and allocating memory for each.

For $d \leq 5$, exhaustively checking all triangulation orbits against all
sign distributions is fast.
For degree six, the full process over all triangulation orbits is just
feasible at roughly $10\,000$ GPU-hours but unlikely to produce valuable insights.
Restricting to equivalence classes containing a symmetric triangulation makes this feasible.
Through these searches, all nonempty real schemes are realized as T-curves for
$d \leq 5$ (a single triangulation suffices), and two
triangulations suffice to cover all $55$ nonempty real schemes of degree
six~\cite[Theorem 20]{2026:GeiselmannEtAl:121PatchworkedDeg7}.

For degree seven, exhaustive sign distribution sweeps over four hand-designed triangulations realize all $121$ nonempty real schemes~\cite[Theorem 24]{2026:GeiselmannEtAl:121PatchworkedDeg7}.
One of these triangulations already produces $115$ of the $121$ real schemes.

For degree eight, prior work has focused on real schemes with the maximum
number of ovals given by Harnack's bound ($22$ ovals for $d = 8$).
At most $89$ are realizable~\cite{2002:Orevkov:FlexibleMcurvesDeg8}, of which
$83$ have algebraic constructions, leaving six open.
Beyond the maximum case, no systematic classification of degree eight real
schemes exists.
Our approach combined two strategies to build an initial pool of triangulations and realized types: simulated annealing searches over triangulations, optimizing for a target $(p,n)$-count or for the number of distinct types realized, and hand-designed triangulations guided by a curated pool of sign distributions.
We then iterated: exhaustively check all sign distributions for the selected triangulations and extract one witness sign distribution per newly realized type, then sweep all symmetric triangulations against the accumulated witnesses. Repeat until convergence.
This iterative process realized exactly $2359$ distinct real schemes with representatives at
every oval count from $1$ to $22$, of which only $30$ are maximum.
Among maximum real schemes, the six left open by Orevkov remain unresolved,
and $53$ others have no known T-curve realization.

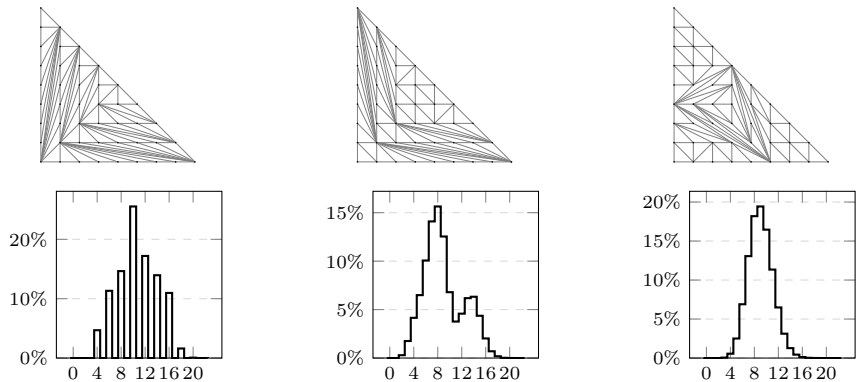
\begin{figure}[t]
\centering
\begin{minipage}[t]{0.31\textwidth}\centering
  \hspace{0.5em}\input{tikz/bow_tie_wireframe}
\end{minipage}\hfill
\begin{minipage}[t]{0.31\textwidth}\centering
  \hspace{0.5em}\input{tikz/rank016_wireframe}
\end{minipage}\hfill
\begin{minipage}[t]{0.31\textwidth}\centering
  \hspace{0.5em}\input{tikz/rank015_wireframe}
\end{minipage}
\par\vspace{8pt}
\begin{minipage}[t]{0.31\textwidth}\centering
  \input{tikz/bow_tie_histogram}
\end{minipage}\hfill
\begin{minipage}[t]{0.31\textwidth}\centering
  \input{tikz/rank016_histogram}
\end{minipage}\hfill
\begin{minipage}[t]{0.31\textwidth}\centering
  \input{tikz/rank015_histogram}
\end{minipage}
\caption{Three degree-eight triangulations of $A=8\cdot\Delta_{2}\cap\ZZ^2$ (top, first
  quadrant shown) and the distribution of the number of ovals over all
  $2^{42}$ sign distributions (bottom).
  The first is the bow tie triangulation~\cite[Section~4.3]{2026:GeiselmannEtAl:121PatchworkedDeg7}, which realizes only even oval counts.
  The three triangulations realize $123$, $359$, and $353$ real schemes,
  respectively.}\label{fig:distributions}
\end{figure}

Since Algorithm~\ref{alg:real-scheme} classifies each pair $(\cT, \sigma)$
independently, exhaustive enumeration is embarrassingly parallel, and since
all data structures are degree-bounded and fit in thread-local memory,
it is naturally suited to GPU execution.
We implemented both a minimal, single-patchwork-focused C++ library
(\texttt{libisotopy}) and a
companion Rust implementation focused on throughput for large exhaustive
searches across one or more GPUs, with both CUDA and Metal backends.
At scale, orchestration overhead --- memory transfer, histogram aggregation,
kernel dispatch --- and data collection become the computationally challenging
aspects rather than the classification itself.
At the compute floor (all overhead amortized), a single NVIDIA A100 classifies
roughly $10^8$ pairs per second for degree eight, with the exact rate
depending on the triangulation.
An exhaustive sweep over all $2^{42}$ sign distributions for one
degree-eight triangulation thus takes on the order of hours on a single GPU;
multi-GPU execution scales linearly.

\section{Conclusion}\label{sec:conclusion}

We presented an algorithm for computing the real scheme of a combinatorial
patchwork in near-quadratic time.
The key ingredients are a union-find computation on the reflected signed
triangulation, a bipartiteness check on the antipodal component graph for root
identification, and the eightfold symmetry of the sign distribution space that makes
exhaustive sweeps tractable.
Two implementations are provided: \texttt{libisotopy}, a publicly available
C++ library, and a multi-GPU Rust implementation with
CUDA and Metal backends.

Two directions stand out for future work.
First, our algorithm computes the real scheme but not the finer
\emph{complex scheme}, which additionally records dividing type and complex
orientations of ovals; algorithms were described by
De~Loera and Wicklin~\cite{1998:DeLoera-Wicklin} and Kaihnsa et al.\ \cite{2019:KaihnsaEtAl:64curves}; but neither invariant seems to have a
general-purpose implementation.
Second, a central open question is whether degree eight is where the
limitations of T-curves become apparent.
Our results so far point in this direction: we realize only $30$ of $89$
maximum real schemes as T-curves, and the gap between the $2359$ real schemes
found and the (unknown) total may be substantial --- a qualitative departure
from degrees through seven, where every nonempty real scheme is a T-curve.
Our exploration of the patchwork search space is far from exhaustive, however,
and future searches may cover more.
Even if T-curves fall short, constructive approaches beyond classical
patchworking may provide realizations for some or all of the remaining cases.

\begin{credits}
\subsubsection{\ackname}
Funded by the Deutsche Forschungsgemeinschaft (DFG, German Research
Foundation) under Germany's Excellence Strategy -- ``The Berlin Mathematics
Research Center MATH+'' (EXC-2046/1, EXC-2046/2, project ID 390685689),
``Symbolic Tools in Mathematics and their Application'' (TRR~195, project ID
286237555), ``Mathematical Modelling, Simulation and Optimization Using the
Example of Gas Networks'' (SFB/TRR~154, project ID 239904186),
``Mathematical Research Data Initiative (MaRDI)'' (project ID 460135501) as
well as by the German Federal Ministry of Research, Technology and Space
(Research Campus MODAL, fund number 05M14ZAM, 05M20ZBM) and the VDI/VDE
Innovation + Technik GmbH (fund number 16IS23025B).
\end{credits}

\printbibliography

\end{document}

%% file: shared-macros.tex
\newcommand{\RR}{\mathbb{R}}
\newcommand{\PP}{\mathbb{P}}
\newcommand{\ZZ}{\mathbb{Z}}

\newcommand{\GF}[1]{\mathbb{F}_{#1}}

\newcommand{\RP}{\RR\PP}

\newcommand{\cC}{\mathcal{C}}

\newcommand{\cS}{\mathcal{S}}
\newcommand{\cT}{\mathcal{T}}

\DeclareMathOperator\conv{conv}   

\newcommand\scheme[1]{\langle #1 \rangle}
\newcommand\Sym[1]{\mathfrak{S}_{#1}}       

\newcommand{\polymake}{\texttt{polymake}\xspace}


%% file: tikz/bow_tie_wireframe.tex
\begin{tikzpicture}[scale=0.256]
  \draw[black!50, line width=0.3pt] (0,0) -- (0,1);
  \draw[black!50, line width=0.3pt] (0,0) -- (1,0);
  \draw[black!50, line width=0.3pt] (0,0) -- (1,1);
  \draw[black!50, line width=0.3pt] (0,0) -- (1,2);
  \draw[black!50, line width=0.3pt] (0,0) -- (1,3);
  \draw[black!50, line width=0.3pt] (0,0) -- (1,4);
  \draw[black!50, line width=0.3pt] (0,0) -- (1,5);
  \draw[black!50, line width=0.3pt] (0,0) -- (1,6);
  \draw[black!50, line width=0.3pt] (0,0) -- (1,7);
  \draw[black!50, line width=0.3pt] (0,1) -- (0,2);
  \draw[black!50, line width=0.3pt] (0,1) -- (1,7);
  \draw[black!50, line width=0.3pt] (0,2) -- (0,3);
  \draw[black!50, line width=0.3pt] (0,2) -- (1,7);
  \draw[black!50, line width=0.3pt] (0,3) -- (0,4);
  \draw[black!50, line width=0.3pt] (0,3) -- (1,7);
  \draw[black!50, line width=0.3pt] (0,4) -- (0,5);
  \draw[black!50, line width=0.3pt] (0,4) -- (1,7);
  \draw[black!50, line width=0.3pt] (0,5) -- (0,6);
  \draw[black!50, line width=0.3pt] (0,5) -- (1,7);
  \draw[black!50, line width=0.3pt] (0,6) -- (0,7);
  \draw[black!50, line width=0.3pt] (0,6) -- (1,7);
  \draw[black!50, line width=0.3pt] (0,7) -- (0,8);
  \draw[black!50, line width=0.3pt] (0,7) -- (1,7);
  \draw[black!50, line width=0.3pt] (0,8) -- (1,7);
  \draw[black!50, line width=0.3pt] (1,0) -- (1,1);
  \draw[black!50, line width=0.3pt] (1,0) -- (2,0);
  \draw[black!50, line width=0.3pt] (1,1) -- (1,2);
  \draw[black!50, line width=0.3pt] (1,1) -- (2,0);
  \draw[black!50, line width=0.3pt] (1,1) -- (2,1);
  \draw[black!50, line width=0.3pt] (1,1) -- (2,2);
  \draw[black!50, line width=0.3pt] (1,1) -- (2,3);
  \draw[black!50, line width=0.3pt] (1,1) -- (2,4);
  \draw[black!50, line width=0.3pt] (1,1) -- (2,5);
  \draw[black!50, line width=0.3pt] (1,1) -- (2,6);
  \draw[black!50, line width=0.3pt] (1,1) -- (3,0);
  \draw[black!50, line width=0.3pt] (1,1) -- (4,0);
  \draw[black!50, line width=0.3pt] (1,1) -- (5,0);
  \draw[black!50, line width=0.3pt] (1,1) -- (6,0);
  \draw[black!50, line width=0.3pt] (1,1) -- (7,0);
  \draw[black!50, line width=0.3pt] (1,1) -- (8,0);
  \draw[black!50, line width=0.3pt] (1,2) -- (1,3);
  \draw[black!50, line width=0.3pt] (1,2) -- (2,6);
  \draw[black!50, line width=0.3pt] (1,3) -- (1,4);
  \draw[black!50, line width=0.3pt] (1,3) -- (2,6);
  \draw[black!50, line width=0.3pt] (1,4) -- (1,5);
  \draw[black!50, line width=0.3pt] (1,4) -- (2,6);
  \draw[black!50, line width=0.3pt] (1,5) -- (1,6);
  \draw[black!50, line width=0.3pt] (1,5) -- (2,6);
  \draw[black!50, line width=0.3pt] (1,6) -- (1,7);
  \draw[black!50, line width=0.3pt] (1,6) -- (2,6);
  \draw[black!50, line width=0.3pt] (1,7) -- (2,6);
  \draw[black!50, line width=0.3pt] (2,0) -- (3,0);
  \draw[black!50, line width=0.3pt] (2,1) -- (2,2);
  \draw[black!50, line width=0.3pt] (2,1) -- (3,1);
  \draw[black!50, line width=0.3pt] (2,1) -- (8,0);
  \draw[black!50, line width=0.3pt] (2,2) -- (2,3);
  \draw[black!50, line width=0.3pt] (2,2) -- (3,1);
  \draw[black!50, line width=0.3pt] (2,2) -- (3,2);
  \draw[black!50, line width=0.3pt] (2,2) -- (3,3);
  \draw[black!50, line width=0.3pt] (2,2) -- (3,4);
  \draw[black!50, line width=0.3pt] (2,2) -- (3,5);
  \draw[black!50, line width=0.3pt] (2,2) -- (4,1);
  \draw[black!50, line width=0.3pt] (2,2) -- (5,1);
  \draw[black!50, line width=0.3pt] (2,2) -- (6,1);
  \draw[black!50, line width=0.3pt] (2,2) -- (7,1);
  \draw[black!50, line width=0.3pt] (2,3) -- (2,4);
  \draw[black!50, line width=0.3pt] (2,3) -- (3,5);
  \draw[black!50, line width=0.3pt] (2,4) -- (2,5);
  \draw[black!50, line width=0.3pt] (2,4) -- (3,5);
  \draw[black!50, line width=0.3pt] (2,5) -- (2,6);
  \draw[black!50, line width=0.3pt] (2,5) -- (3,5);
  \draw[black!50, line width=0.3pt] (2,6) -- (3,5);
  \draw[black!50, line width=0.3pt] (3,0) -- (4,0);
  \draw[black!50, line width=0.3pt] (3,1) -- (4,1);
  \draw[black!50, line width=0.3pt] (3,1) -- (8,0);
  \draw[black!50, line width=0.3pt] (3,2) -- (3,3);
  \draw[black!50, line width=0.3pt] (3,2) -- (4,2);
  \draw[black!50, line width=0.3pt] (3,2) -- (7,1);
  \draw[black!50, line width=0.3pt] (3,3) -- (3,4);
  \draw[black!50, line width=0.3pt] (3,3) -- (4,2);
  \draw[black!50, line width=0.3pt] (3,3) -- (4,3);
  \draw[black!50, line width=0.3pt] (3,3) -- (4,4);
  \draw[black!50, line width=0.3pt] (3,3) -- (5,2);
  \draw[black!50, line width=0.3pt] (3,3) -- (6,2);
  \draw[black!50, line width=0.3pt] (3,4) -- (3,5);
  \draw[black!50, line width=0.3pt] (3,4) -- (4,4);
  \draw[black!50, line width=0.3pt] (3,5) -- (4,4);
  \draw[black!50, line width=0.3pt] (4,0) -- (5,0);
  \draw[black!50, line width=0.3pt] (4,1) -- (5,1);
  \draw[black!50, line width=0.3pt] (4,1) -- (8,0);
  \draw[black!50, line width=0.3pt] (4,2) -- (5,2);
  \draw[black!50, line width=0.3pt] (4,2) -- (7,1);
  \draw[black!50, line width=0.3pt] (4,3) -- (4,4);
  \draw[black!50, line width=0.3pt] (4,3) -- (5,3);
  \draw[black!50, line width=0.3pt] (4,3) -- (6,2);
  \draw[black!50, line width=0.3pt] (4,4) -- (5,3);
  \draw[black!50, line width=0.3pt] (5,0) -- (6,0);
  \draw[black!50, line width=0.3pt] (5,1) -- (6,1);
  \draw[black!50, line width=0.3pt] (5,1) -- (8,0);
  \draw[black!50, line width=0.3pt] (5,2) -- (6,2);
  \draw[black!50, line width=0.3pt] (5,2) -- (7,1);
  \draw[black!50, line width=0.3pt] (5,3) -- (6,2);
  \draw[black!50, line width=0.3pt] (6,0) -- (7,0);
  \draw[black!50, line width=0.3pt] (6,1) -- (7,1);
  \draw[black!50, line width=0.3pt] (6,1) -- (8,0);
  \draw[black!50, line width=0.3pt] (6,2) -- (7,1);
  \draw[black!50, line width=0.3pt] (7,0) -- (8,0);
  \draw[black!50, line width=0.3pt] (7,1) -- (8,0);
  \fill[black] (0,0) circle (1.2pt);
  \fill[black] (0,1) circle (1.2pt);
  \fill[black] (0,2) circle (1.2pt);
  \fill[black] (0,3) circle (1.2pt);
  \fill[black] (0,4) circle (1.2pt);
  \fill[black] (0,5) circle (1.2pt);
  \fill[black] (0,6) circle (1.2pt);
  \fill[black] (0,7) circle (1.2pt);
  \fill[black] (0,8) circle (1.2pt);
  \fill[black] (1,0) circle (1.2pt);
  \fill[black] (1,1) circle (1.2pt);
  \fill[black] (1,2) circle (1.2pt);
  \fill[black] (1,3) circle (1.2pt);
  \fill[black] (1,4) circle (1.2pt);
  \fill[black] (1,5) circle (1.2pt);
  \fill[black] (1,6) circle (1.2pt);
  \fill[black] (1,7) circle (1.2pt);
  \fill[black] (2,0) circle (1.2pt);
  \fill[black] (2,1) circle (1.2pt);
  \fill[black] (2,2) circle (1.2pt);
  \fill[black] (2,3) circle (1.2pt);
  \fill[black] (2,4) circle (1.2pt);
  \fill[black] (2,5) circle (1.2pt);
  \fill[black] (2,6) circle (1.2pt);
  \fill[black] (3,0) circle (1.2pt);
  \fill[black] (3,1) circle (1.2pt);
  \fill[black] (3,2) circle (1.2pt);
  \fill[black] (3,3) circle (1.2pt);
  \fill[black] (3,4) circle (1.2pt);
  \fill[black] (3,5) circle (1.2pt);
  \fill[black] (4,0) circle (1.2pt);
  \fill[black] (4,1) circle (1.2pt);
  \fill[black] (4,2) circle (1.2pt);
  \fill[black] (4,3) circle (1.2pt);
  \fill[black] (4,4) circle (1.2pt);
  \fill[black] (5,0) circle (1.2pt);
  \fill[black] (5,1) circle (1.2pt);
  \fill[black] (5,2) circle (1.2pt);
  \fill[black] (5,3) circle (1.2pt);
  \fill[black] (6,0) circle (1.2pt);
  \fill[black] (6,1) circle (1.2pt);
  \fill[black] (6,2) circle (1.2pt);
  \fill[black] (7,0) circle (1.2pt);
  \fill[black] (7,1) circle (1.2pt);
  \fill[black] (8,0) circle (1.2pt);
\end{tikzpicture}

%% file: tikz/rank016_wireframe.tex
\begin{tikzpicture}[scale=0.256]
  \draw[black!50, line width=0.3pt] (0,0) -- (0,1);
  \draw[black!50, line width=0.3pt] (0,0) -- (1,0);
  \draw[black!50, line width=0.3pt] (0,1) -- (0,2);
  \draw[black!50, line width=0.3pt] (0,1) -- (1,0);
  \draw[black!50, line width=0.3pt] (0,1) -- (1,1);
  \draw[black!50, line width=0.3pt] (0,2) -- (0,3);
  \draw[black!50, line width=0.3pt] (0,2) -- (1,1);
  \draw[black!50, line width=0.3pt] (0,3) -- (0,4);
  \draw[black!50, line width=0.3pt] (0,3) -- (1,1);
  \draw[black!50, line width=0.3pt] (0,4) -- (0,5);
  \draw[black!50, line width=0.3pt] (0,4) -- (1,1);
  \draw[black!50, line width=0.3pt] (0,5) -- (0,6);
  \draw[black!50, line width=0.3pt] (0,5) -- (1,1);
  \draw[black!50, line width=0.3pt] (0,6) -- (0,7);
  \draw[black!50, line width=0.3pt] (0,6) -- (1,1);
  \draw[black!50, line width=0.3pt] (0,7) -- (0,8);
  \draw[black!50, line width=0.3pt] (0,7) -- (1,1);
  \draw[black!50, line width=0.3pt] (0,8) -- (1,1);
  \draw[black!50, line width=0.3pt] (0,8) -- (1,2);
  \draw[black!50, line width=0.3pt] (0,8) -- (1,3);
  \draw[black!50, line width=0.3pt] (0,8) -- (1,4);
  \draw[black!50, line width=0.3pt] (0,8) -- (1,5);
  \draw[black!50, line width=0.3pt] (0,8) -- (1,6);
  \draw[black!50, line width=0.3pt] (0,8) -- (1,7);
  \draw[black!50, line width=0.3pt] (1,0) -- (1,1);
  \draw[black!50, line width=0.3pt] (1,0) -- (2,0);
  \draw[black!50, line width=0.3pt] (1,1) -- (1,2);
  \draw[black!50, line width=0.3pt] (1,1) -- (2,0);
  \draw[black!50, line width=0.3pt] (1,1) -- (2,1);
  \draw[black!50, line width=0.3pt] (1,1) -- (3,0);
  \draw[black!50, line width=0.3pt] (1,1) -- (4,0);
  \draw[black!50, line width=0.3pt] (1,1) -- (5,0);
  \draw[black!50, line width=0.3pt] (1,1) -- (6,0);
  \draw[black!50, line width=0.3pt] (1,1) -- (7,0);
  \draw[black!50, line width=0.3pt] (1,1) -- (8,0);
  \draw[black!50, line width=0.3pt] (1,2) -- (1,3);
  \draw[black!50, line width=0.3pt] (1,2) -- (2,1);
  \draw[black!50, line width=0.3pt] (1,2) -- (2,2);
  \draw[black!50, line width=0.3pt] (1,3) -- (1,4);
  \draw[black!50, line width=0.3pt] (1,3) -- (2,2);
  \draw[black!50, line width=0.3pt] (1,4) -- (1,5);
  \draw[black!50, line width=0.3pt] (1,4) -- (2,2);
  \draw[black!50, line width=0.3pt] (1,5) -- (1,6);
  \draw[black!50, line width=0.3pt] (1,5) -- (2,2);
  \draw[black!50, line width=0.3pt] (1,6) -- (1,7);
  \draw[black!50, line width=0.3pt] (1,6) -- (2,2);
  \draw[black!50, line width=0.3pt] (1,7) -- (2,2);
  \draw[black!50, line width=0.3pt] (1,7) -- (2,3);
  \draw[black!50, line width=0.3pt] (1,7) -- (2,4);
  \draw[black!50, line width=0.3pt] (1,7) -- (2,5);
  \draw[black!50, line width=0.3pt] (1,7) -- (2,6);
  \draw[black!50, line width=0.3pt] (2,0) -- (3,0);
  \draw[black!50, line width=0.3pt] (2,1) -- (2,2);
  \draw[black!50, line width=0.3pt] (2,1) -- (3,1);
  \draw[black!50, line width=0.3pt] (2,1) -- (8,0);
  \draw[black!50, line width=0.3pt] (2,2) -- (2,3);
  \draw[black!50, line width=0.3pt] (2,2) -- (3,1);
  \draw[black!50, line width=0.3pt] (2,2) -- (3,2);
  \draw[black!50, line width=0.3pt] (2,2) -- (4,1);
  \draw[black!50, line width=0.3pt] (2,2) -- (5,1);
  \draw[black!50, line width=0.3pt] (2,2) -- (6,1);
  \draw[black!50, line width=0.3pt] (2,2) -- (7,1);
  \draw[black!50, line width=0.3pt] (2,3) -- (2,4);
  \draw[black!50, line width=0.3pt] (2,3) -- (3,2);
  \draw[black!50, line width=0.3pt] (2,3) -- (3,3);
  \draw[black!50, line width=0.3pt] (2,4) -- (2,5);
  \draw[black!50, line width=0.3pt] (2,4) -- (3,3);
  \draw[black!50, line width=0.3pt] (2,4) -- (3,4);
  \draw[black!50, line width=0.3pt] (2,4) -- (3,5);
  \draw[black!50, line width=0.3pt] (2,5) -- (2,6);
  \draw[black!50, line width=0.3pt] (2,5) -- (3,5);
  \draw[black!50, line width=0.3pt] (2,6) -- (3,5);
  \draw[black!50, line width=0.3pt] (3,0) -- (4,0);
  \draw[black!50, line width=0.3pt] (3,1) -- (4,1);
  \draw[black!50, line width=0.3pt] (3,1) -- (8,0);
  \draw[black!50, line width=0.3pt] (3,2) -- (3,3);
  \draw[black!50, line width=0.3pt] (3,2) -- (4,2);
  \draw[black!50, line width=0.3pt] (3,2) -- (7,1);
  \draw[black!50, line width=0.3pt] (3,3) -- (3,4);
  \draw[black!50, line width=0.3pt] (3,3) -- (4,2);
  \draw[black!50, line width=0.3pt] (3,3) -- (4,3);
  \draw[black!50, line width=0.3pt] (3,4) -- (3,5);
  \draw[black!50, line width=0.3pt] (3,4) -- (4,3);
  \draw[black!50, line width=0.3pt] (3,4) -- (4,4);
  \draw[black!50, line width=0.3pt] (3,5) -- (4,4);
  \draw[black!50, line width=0.3pt] (4,0) -- (5,0);
  \draw[black!50, line width=0.3pt] (4,1) -- (5,1);
  \draw[black!50, line width=0.3pt] (4,1) -- (8,0);
  \draw[black!50, line width=0.3pt] (4,2) -- (4,3);
  \draw[black!50, line width=0.3pt] (4,2) -- (5,2);
  \draw[black!50, line width=0.3pt] (4,2) -- (5,3);
  \draw[black!50, line width=0.3pt] (4,2) -- (7,1);
  \draw[black!50, line width=0.3pt] (4,3) -- (4,4);
  \draw[black!50, line width=0.3pt] (4,3) -- (5,3);
  \draw[black!50, line width=0.3pt] (4,4) -- (5,3);
  \draw[black!50, line width=0.3pt] (5,0) -- (6,0);
  \draw[black!50, line width=0.3pt] (5,1) -- (6,1);
  \draw[black!50, line width=0.3pt] (5,1) -- (8,0);
  \draw[black!50, line width=0.3pt] (5,2) -- (5,3);
  \draw[black!50, line width=0.3pt] (5,2) -- (6,2);
  \draw[black!50, line width=0.3pt] (5,2) -- (7,1);
  \draw[black!50, line width=0.3pt] (5,3) -- (6,2);
  \draw[black!50, line width=0.3pt] (6,0) -- (7,0);
  \draw[black!50, line width=0.3pt] (6,1) -- (7,1);
  \draw[black!50, line width=0.3pt] (6,1) -- (8,0);
  \draw[black!50, line width=0.3pt] (6,2) -- (7,1);
  \draw[black!50, line width=0.3pt] (7,0) -- (8,0);
  \draw[black!50, line width=0.3pt] (7,1) -- (8,0);
  \fill[black] (0,0) circle (1.2pt);
  \fill[black] (0,1) circle (1.2pt);
  \fill[black] (0,2) circle (1.2pt);
  \fill[black] (0,3) circle (1.2pt);
  \fill[black] (0,4) circle (1.2pt);
  \fill[black] (0,5) circle (1.2pt);
  \fill[black] (0,6) circle (1.2pt);
  \fill[black] (0,7) circle (1.2pt);
  \fill[black] (0,8) circle (1.2pt);
  \fill[black] (1,0) circle (1.2pt);
  \fill[black] (1,1) circle (1.2pt);
  \fill[black] (1,2) circle (1.2pt);
  \fill[black] (1,3) circle (1.2pt);
  \fill[black] (1,4) circle (1.2pt);
  \fill[black] (1,5) circle (1.2pt);
  \fill[black] (1,6) circle (1.2pt);
  \fill[black] (1,7) circle (1.2pt);
  \fill[black] (2,0) circle (1.2pt);
  \fill[black] (2,1) circle (1.2pt);
  \fill[black] (2,2) circle (1.2pt);
  \fill[black] (2,3) circle (1.2pt);
  \fill[black] (2,4) circle (1.2pt);
  \fill[black] (2,5) circle (1.2pt);
  \fill[black] (2,6) circle (1.2pt);
  \fill[black] (3,0) circle (1.2pt);
  \fill[black] (3,1) circle (1.2pt);
  \fill[black] (3,2) circle (1.2pt);
  \fill[black] (3,3) circle (1.2pt);
  \fill[black] (3,4) circle (1.2pt);
  \fill[black] (3,5) circle (1.2pt);
  \fill[black] (4,0) circle (1.2pt);
  \fill[black] (4,1) circle (1.2pt);
  \fill[black] (4,2) circle (1.2pt);
  \fill[black] (4,3) circle (1.2pt);
  \fill[black] (4,4) circle (1.2pt);
  \fill[black] (5,0) circle (1.2pt);
  \fill[black] (5,1) circle (1.2pt);
  \fill[black] (5,2) circle (1.2pt);
  \fill[black] (5,3) circle (1.2pt);
  \fill[black] (6,0) circle (1.2pt);
  \fill[black] (6,1) circle (1.2pt);
  \fill[black] (6,2) circle (1.2pt);
  \fill[black] (7,0) circle (1.2pt);
  \fill[black] (7,1) circle (1.2pt);
  \fill[black] (8,0) circle (1.2pt);
\end{tikzpicture}

%% file: tikz/rank015_wireframe.tex
\begin{tikzpicture}[scale=0.256]
  \draw[black!50, line width=0.3pt] (0,0) -- (0,1);
  \draw[black!50, line width=0.3pt] (0,0) -- (1,0);
  \draw[black!50, line width=0.3pt] (0,1) -- (0,2);
  \draw[black!50, line width=0.3pt] (0,1) -- (1,0);
  \draw[black!50, line width=0.3pt] (0,1) -- (1,1);
  \draw[black!50, line width=0.3pt] (0,2) -- (0,3);
  \draw[black!50, line width=0.3pt] (0,2) -- (1,1);
  \draw[black!50, line width=0.3pt] (0,2) -- (1,2);
  \draw[black!50, line width=0.3pt] (0,2) -- (2,1);
  \draw[black!50, line width=0.3pt] (0,2) -- (3,1);
  \draw[black!50, line width=0.3pt] (0,3) -- (0,4);
  \draw[black!50, line width=0.3pt] (0,3) -- (1,2);
  \draw[black!50, line width=0.3pt] (0,3) -- (1,3);
  \draw[black!50, line width=0.3pt] (0,3) -- (1,4);
  \draw[black!50, line width=0.3pt] (0,3) -- (2,2);
  \draw[black!50, line width=0.3pt] (0,3) -- (2,4);
  \draw[black!50, line width=0.3pt] (0,3) -- (3,1);
  \draw[black!50, line width=0.3pt] (0,3) -- (3,4);
  \draw[black!50, line width=0.3pt] (0,3) -- (3,5);
  \draw[black!50, line width=0.3pt] (0,3) -- (5,0);
  \draw[black!50, line width=0.3pt] (0,4) -- (0,5);
  \draw[black!50, line width=0.3pt] (0,4) -- (1,4);
  \draw[black!50, line width=0.3pt] (0,5) -- (0,6);
  \draw[black!50, line width=0.3pt] (0,5) -- (1,4);
  \draw[black!50, line width=0.3pt] (0,5) -- (1,5);
  \draw[black!50, line width=0.3pt] (0,6) -- (0,7);
  \draw[black!50, line width=0.3pt] (0,6) -- (1,5);
  \draw[black!50, line width=0.3pt] (0,6) -- (1,6);
  \draw[black!50, line width=0.3pt] (0,7) -- (0,8);
  \draw[black!50, line width=0.3pt] (0,7) -- (1,6);
  \draw[black!50, line width=0.3pt] (0,7) -- (1,7);
  \draw[black!50, line width=0.3pt] (0,8) -- (1,7);
  \draw[black!50, line width=0.3pt] (1,0) -- (1,1);
  \draw[black!50, line width=0.3pt] (1,0) -- (2,0);
  \draw[black!50, line width=0.3pt] (1,0) -- (2,1);
  \draw[black!50, line width=0.3pt] (1,1) -- (2,1);
  \draw[black!50, line width=0.3pt] (1,2) -- (3,1);
  \draw[black!50, line width=0.3pt] (1,3) -- (2,2);
  \draw[black!50, line width=0.3pt] (1,3) -- (2,3);
  \draw[black!50, line width=0.3pt] (1,3) -- (3,2);
  \draw[black!50, line width=0.3pt] (1,3) -- (3,4);
  \draw[black!50, line width=0.3pt] (1,3) -- (4,1);
  \draw[black!50, line width=0.3pt] (1,3) -- (5,0);
  \draw[black!50, line width=0.3pt] (1,4) -- (1,5);
  \draw[black!50, line width=0.3pt] (1,4) -- (2,5);
  \draw[black!50, line width=0.3pt] (1,4) -- (3,5);
  \draw[black!50, line width=0.3pt] (1,5) -- (1,6);
  \draw[black!50, line width=0.3pt] (1,5) -- (2,5);
  \draw[black!50, line width=0.3pt] (1,6) -- (1,7);
  \draw[black!50, line width=0.3pt] (1,6) -- (2,5);
  \draw[black!50, line width=0.3pt] (1,6) -- (2,6);
  \draw[black!50, line width=0.3pt] (1,7) -- (2,6);
  \draw[black!50, line width=0.3pt] (2,0) -- (2,1);
  \draw[black!50, line width=0.3pt] (2,0) -- (3,0);
  \draw[black!50, line width=0.3pt] (2,1) -- (3,0);
  \draw[black!50, line width=0.3pt] (2,1) -- (3,1);
  \draw[black!50, line width=0.3pt] (2,2) -- (5,0);
  \draw[black!50, line width=0.3pt] (2,3) -- (3,2);
  \draw[black!50, line width=0.3pt] (2,3) -- (3,3);
  \draw[black!50, line width=0.3pt] (2,3) -- (3,4);
  \draw[black!50, line width=0.3pt] (2,4) -- (3,4);
  \draw[black!50, line width=0.3pt] (2,4) -- (3,5);
  \draw[black!50, line width=0.3pt] (2,5) -- (2,6);
  \draw[black!50, line width=0.3pt] (2,5) -- (3,5);
  \draw[black!50, line width=0.3pt] (2,6) -- (3,5);
  \draw[black!50, line width=0.3pt] (3,0) -- (3,1);
  \draw[black!50, line width=0.3pt] (3,0) -- (4,0);
  \draw[black!50, line width=0.3pt] (3,1) -- (4,0);
  \draw[black!50, line width=0.3pt] (3,1) -- (5,0);
  \draw[black!50, line width=0.3pt] (3,2) -- (3,3);
  \draw[black!50, line width=0.3pt] (3,2) -- (4,1);
  \draw[black!50, line width=0.3pt] (3,3) -- (3,4);
  \draw[black!50, line width=0.3pt] (3,3) -- (4,1);
  \draw[black!50, line width=0.3pt] (3,4) -- (3,5);
  \draw[black!50, line width=0.3pt] (3,4) -- (4,1);
  \draw[black!50, line width=0.3pt] (3,5) -- (4,1);
  \draw[black!50, line width=0.3pt] (3,5) -- (4,2);
  \draw[black!50, line width=0.3pt] (3,5) -- (4,3);
  \draw[black!50, line width=0.3pt] (3,5) -- (4,4);
  \draw[black!50, line width=0.3pt] (3,5) -- (5,0);
  \draw[black!50, line width=0.3pt] (4,0) -- (5,0);
  \draw[black!50, line width=0.3pt] (4,1) -- (4,2);
  \draw[black!50, line width=0.3pt] (4,1) -- (5,0);
  \draw[black!50, line width=0.3pt] (4,2) -- (5,0);
  \draw[black!50, line width=0.3pt] (4,3) -- (4,4);
  \draw[black!50, line width=0.3pt] (4,3) -- (5,0);
  \draw[black!50, line width=0.3pt] (4,3) -- (5,1);
  \draw[black!50, line width=0.3pt] (4,3) -- (5,2);
  \draw[black!50, line width=0.3pt] (4,4) -- (5,2);
  \draw[black!50, line width=0.3pt] (4,4) -- (5,3);
  \draw[black!50, line width=0.3pt] (5,0) -- (5,1);
  \draw[black!50, line width=0.3pt] (5,0) -- (6,0);
  \draw[black!50, line width=0.3pt] (5,1) -- (5,2);
  \draw[black!50, line width=0.3pt] (5,1) -- (6,0);
  \draw[black!50, line width=0.3pt] (5,1) -- (6,1);
  \draw[black!50, line width=0.3pt] (5,2) -- (5,3);
  \draw[black!50, line width=0.3pt] (5,2) -- (6,1);
  \draw[black!50, line width=0.3pt] (5,2) -- (6,2);
  \draw[black!50, line width=0.3pt] (5,3) -- (6,2);
  \draw[black!50, line width=0.3pt] (6,0) -- (6,1);
  \draw[black!50, line width=0.3pt] (6,0) -- (7,0);
  \draw[black!50, line width=0.3pt] (6,1) -- (6,2);
  \draw[black!50, line width=0.3pt] (6,1) -- (7,0);
  \draw[black!50, line width=0.3pt] (6,1) -- (7,1);
  \draw[black!50, line width=0.3pt] (6,2) -- (7,1);
  \draw[black!50, line width=0.3pt] (7,0) -- (7,1);
  \draw[black!50, line width=0.3pt] (7,0) -- (8,0);
  \draw[black!50, line width=0.3pt] (7,1) -- (8,0);
  \fill[black] (0,0) circle (1.2pt);
  \fill[black] (0,1) circle (1.2pt);
  \fill[black] (0,2) circle (1.2pt);
  \fill[black] (0,3) circle (1.2pt);
  \fill[black] (0,4) circle (1.2pt);
  \fill[black] (0,5) circle (1.2pt);
  \fill[black] (0,6) circle (1.2pt);
  \fill[black] (0,7) circle (1.2pt);
  \fill[black] (0,8) circle (1.2pt);
  \fill[black] (1,0) circle (1.2pt);
  \fill[black] (1,1) circle (1.2pt);
  \fill[black] (1,2) circle (1.2pt);
  \fill[black] (1,3) circle (1.2pt);
  \fill[black] (1,4) circle (1.2pt);
  \fill[black] (1,5) circle (1.2pt);
  \fill[black] (1,6) circle (1.2pt);
  \fill[black] (1,7) circle (1.2pt);
  \fill[black] (2,0) circle (1.2pt);
  \fill[black] (2,1) circle (1.2pt);
  \fill[black] (2,2) circle (1.2pt);
  \fill[black] (2,3) circle (1.2pt);
  \fill[black] (2,4) circle (1.2pt);
  \fill[black] (2,5) circle (1.2pt);
  \fill[black] (2,6) circle (1.2pt);
  \fill[black] (3,0) circle (1.2pt);
  \fill[black] (3,1) circle (1.2pt);
  \fill[black] (3,2) circle (1.2pt);
  \fill[black] (3,3) circle (1.2pt);
  \fill[black] (3,4) circle (1.2pt);
  \fill[black] (3,5) circle (1.2pt);
  \fill[black] (4,0) circle (1.2pt);
  \fill[black] (4,1) circle (1.2pt);
  \fill[black] (4,2) circle (1.2pt);
  \fill[black] (4,3) circle (1.2pt);
  \fill[black] (4,4) circle (1.2pt);
  \fill[black] (5,0) circle (1.2pt);
  \fill[black] (5,1) circle (1.2pt);
  \fill[black] (5,2) circle (1.2pt);
  \fill[black] (5,3) circle (1.2pt);
  \fill[black] (6,0) circle (1.2pt);
  \fill[black] (6,1) circle (1.2pt);
  \fill[black] (6,2) circle (1.2pt);
  \fill[black] (7,0) circle (1.2pt);
  \fill[black] (7,1) circle (1.2pt);
  \fill[black] (8,0) circle (1.2pt);
\end{tikzpicture}

%% file: tikz/bow_tie_histogram.tex
\begin{tikzpicture}
\begin{axis}[
    xtick={0,4,8,12,16,20},
    tick label style={font=\scriptsize},
    ymajorgrids=true,
    grid style={dashed,gray!30},
    width=\linewidth,
    height=3.8cm,
    ymin=0,
    yticklabel={\pgfmathprintnumber{\tick}\%},
    yticklabel style={font=\scriptsize, /pgf/number format/.cd, fixed, fixed zerofill, precision=0},
]
\addplot[const plot, draw=black, thick, no markers] coordinates {
    (-0.5,0.0000)
    (0.5,0.0000)
    (1.5,0.0000)
    (2.5,0.0000)
    (3.5,4.6875)
    (4.5,0.0000)
    (5.5,11.3281)
    (6.5,0.0000)
    (7.5,14.6484)
    (8.5,0.0000)
    (9.5,25.5371)
    (10.5,0.0000)
    (11.5,17.1997)
    (12.5,0.0000)
    (13.5,13.9526)
    (14.5,0.0000)
    (15.5,10.9695)
    (16.5,0.0000)
    (17.5,1.6113)
    (18.5,0.0000)
    (19.5,0.0652)
    (20.5,0.0000)
    (21.5,0.0004)
    (22.5,0)
};
\end{axis}
\end{tikzpicture}

%% file: tikz/rank016_histogram.tex
\begin{tikzpicture}
\begin{axis}[
    xtick={0,4,8,12,16,20},
    tick label style={font=\scriptsize},
    ymajorgrids=true,
    grid style={dashed,gray!30},
    width=\linewidth,
    height=3.8cm,
    ymin=0,
    yticklabel={\pgfmathprintnumber{\tick}\%},
    yticklabel style={font=\scriptsize, /pgf/number format/.cd, fixed, fixed zerofill, precision=0},
]
\addplot[const plot, draw=black, thick, no markers] coordinates {
    (-0.5,0.0000)
    (0.5,0.0000)
    (1.5,0.2930)
    (2.5,1.7578)
    (3.5,4.1504)
    (4.5,6.4941)
    (5.5,10.0647)
    (6.5,14.0991)
    (7.5,15.6433)
    (8.5,12.5443)
    (9.5,6.7974)
    (10.5,3.7655)
    (11.5,4.5805)
    (12.5,6.1824)
    (13.5,6.3260)
    (14.5,4.3626)
    (15.5,2.0241)
    (16.5,0.6985)
    (17.5,0.1847)
    (18.5,0.0253)
    (19.5,0.0064)
    (20.5,0.0000)
    (21.5,0.0000)
    (22.5,0)
};
\end{axis}
\end{tikzpicture}

%% file: tikz/rank015_histogram.tex
\begin{tikzpicture}
\begin{axis}[
    xtick={0,4,8,12,16,20},
    tick label style={font=\scriptsize},
    ymajorgrids=true,
    grid style={dashed,gray!30},
    width=\linewidth,
    height=3.8cm,
    ymin=0,
    yticklabel={\pgfmathprintnumber{\tick}\%},
    yticklabel style={font=\scriptsize, /pgf/number format/.cd, fixed, fixed zerofill, precision=0},
]
\addplot[const plot, draw=black, thick, no markers] coordinates {
    (-0.5,0.0000)
    (0.5,0.0000)
    (1.5,0.0000)
    (2.5,0.0572)
    (3.5,0.5581)
    (4.5,2.5065)
    (5.5,6.8969)
    (6.5,13.0667)
    (7.5,18.1908)
    (8.5,19.4337)
    (9.5,16.4603)
    (10.5,11.3434)
    (11.5,6.4880)
    (12.5,3.1214)
    (13.5,1.2716)
    (14.5,0.4390)
    (15.5,0.1279)
    (16.5,0.0311)
    (17.5,0.0062)
    (18.5,0.0010)
    (19.5,0.0001)
    (20.5,0.0000)
    (21.5,0.0000)
    (22.5,0)
};
\end{axis}
\end{tikzpicture}